\documentclass{cimart}

\usepackage{longtable}

\title{The complete classification of irreducible components of varieties of  Jordan superalgebras}

\author{Renato Fehlberg J\'{u}nior, Ivan Kaygorodov and Azamat Saydaliyev}

\authorinfo[Renato Fehlberg J\'{u}nior]{Federal University of Espirito Santo,  Brazil}{renato.fehlberg@ufes.br}

\authorinfo[Ivan Kaygorodov]{University of  Beira Interior,  Portugal}{kaygorodov.ivan@gmail.com}

\authorinfo[Azamat Saydaliyev]{Institute of Mathematics Academy of Sciences of Uzbekistan,  Uzbekistan}{azamatsaydaliyev6@gmail.com}

\abstract{%
   The aim of the present short note is to answer the open questions posted by Hernández,  Martin, and  Rodrigues  in \cite{p1,p2}. The obtained results give the complete classification of irreducible components in the varieties of Jordan superalgebras of types $(3,1)$ and $(2,2).$
   }

\keywords{Jordan superalgebra, irreducible component, geometric classification.}

\msc{17A70 (primary); 17C70, 17A30, 14D06, 14L30 (secondary).}

\VOLUME{33}
\YEAR{2025}
\ISSUE{3}
\NUMBER{15}
\DOI{https://doi.org/10.46298/cm.16917}
\begin{document}
 

\section{Introduction}

The algebraic classification of small-dimensional algebras has long been a central topic in the theory of non-associative algebras, and recent years have witnessed a growing interest in their geometric classification as well. Numerous results concerning the algebraic and geometric classification of various non-associative algebras can be found in \cite{k23} and \cite{MS}, including references therein. To name a few more recent developments, the geometric classification of symmetric Leibniz algebras was obtained in \cite{FKS}, while the geometric classification of nilpotent Lie–Yamaguti, Bol, and compatible Lie algebras was given by Abdurasulov, Khudoyberdiyev, and Toshtemirova in \cite{kz}.

In this paper, we focus on four-dimensional Jordan superalgebras, whose algebraic and geometric classifications were studied in \cite{aks, p1, p2}. While the authors of \cite{aks} focused exclusively on identifying irreducible components (i.e., on "rigid" superalgebras and "rigid" families of superalgebras), the authors of \cite{p1, p2} attempted to present the full description of irreducible components (i.e., the description of superalgebras that are in the orbit closure of "rigid" superalgebras and "rigid" families of superalgebras) and left several open problems in the varieties of Jordan superalgebras of types $(2,2)$ and $(3,1)$. Our main goal is to solve these problems, thereby completing the  description of irreducible components presented in the abovementioned works. After presenting the necessary preliminaries in the following section, we show in the last section that there is no degeneration between any of the pairs of superalgebras left as open problems in \cite{p1,p2}.

\section{Preliminaries}

All superalgebras in this work are considered over the field of complex numbers $\mathbb{C}$. A commutative algebra is called a {\it  Jordan  algebra}  if it satisfies the identity $(x^2y)x=x^2(yx).$  A \textit{superalgebra} $\mathcal{A}$ is an algebra with a $\mathbb{Z}_2$-grading, i.e. $\mathcal{A}=\mathcal{A}_0\oplus\mathcal{A}_1$ is a direct sum of two vector spaces and
        $$\mathcal{A}_i \mathcal{A}_j \subseteq \mathcal{A}_{i+j}, \ \ \text{where} \ \ i,j\in \mathbb{Z}_2.$$
A \textit{Jordan superalgebra} is a superalgebra $\mathfrak{J}=\mathfrak{J}_0 + \mathfrak{J}_1$ satisfying the graded identities:
  \begin{center}
        ${xy=(-1)^{|x||y|}yx,}$
  \end{center}
\begin{flushleft} $ {\big((xy)z\big)t+(-1)^{|y||z|+|y||t|+|z||t|}\big((xt)z\big)y+(-1)^{|x||y|+|x||z|+|x||t|+|z||t|}\big((yt)z\big)x\ =}$ \end{flushleft}
\begin{flushright}
    ${(xy)(zt)+(-1)^{|t|(|y|+|z|)}(xt)(yz)+(-1)^{|y||z|}(xz)(yt),}$
\end{flushright}
where $|x|=i$ for $x \in \mathfrak{J}_i$.
As the full list of four-dimensional Jordan superalgebras is extensive, we do not reproduce it here and instead refer the reader to \cite{p1, p2}.

Let \( V = V_0 \oplus V_1 \) be a \( \mathbb{Z}_2 \)-graded vector space with a fixed homogeneous basis  
\(\{e_1, \ldots, e_m, f_1, \ldots, f_n\}\). A Jordan superalgebra structure on \(V\) can be described via structure constants 
\((\alpha_{ij}^k, \beta_{ij}^k, \gamma_{ij}^k) \in \mathbb{C}^{m^3 + 2mn^2}\), where the multiplication is defined as:
\[
e_i e_j = \sum_{k=1}^{m} \alpha_{ij}^k e_k, \quad 
e_i f_j = \sum_{k=1}^{n} \beta_{ij}^k f_k, \quad 
f_i f_j = \sum_{k=1}^{m} \gamma_{ij}^k e_k.
\]
These constants must satisfy the supercommutativity and the Jordan superidentity. Thus, the set of all Jordan superalgebras of type \((m,n)\) forms an affine variety in \( \mathbb{C}^{m^3 + 2mn^2} \), denoted by \( \mathcal{JS}^{(m,n)} \). A point \((\alpha_{ij}^k, \beta_{ij}^k, \gamma_{ij}^k) \in \mathcal{JS}^{(m,n)}\) represents a Jordan superalgebra \( \mathcal{J} \) of type \((m,n)\) with respect to the chosen basis.

More generally, let $\mathcal{S}^{m,n}$ denote the set of all superalgebras of dimension $(m,n)$ defined by a family of polynomial superidentities $T$, regarded as a subset $\mathbb{L}(T)$ of the affine variety $\operatorname{Hom}(V \otimes V, V)$. Then $\mathcal{S}^{m,n}$ is a Zariski-closed subset of the variety $\operatorname{Hom}(V \otimes V, V)$. 

The group $G = (\operatorname{Aut} V)_0 \simeq \operatorname{GL}(V_0) \oplus \operatorname{GL}(V_1)$ acts on $\mathcal{S}^{m,n}$ by conjugation:
\[
(g * \mu)(x \otimes y) = g \mu(g^{-1} x \otimes g^{-1} y),
\]
for all $x, y \in V$, $\mu \in \mathbb{L}(T)$, and $g \in G$.

Let $\mathcal{O}(\mu)$ denote the orbit of $\mu \in \mathbb{L}(T)$ under the action of $G$, and let $\overline{\mathcal{O}(\mu)}$ be the Zariski closure of $\mathcal{O}(\mu)$. Suppose $J, J' \in \mathcal{S}^{m,n}$ are represented by $\lambda, \mu \in \mathbb{L}(T)$, respectively. We say that $\lambda$ degenerates to $\mu$, denoted $\lambda \to \mu$, if $\mu \in \overline{\mathcal{O}(\lambda)}$. In this case, we have $\overline{\mathcal{O}(\mu)} \subset \overline{\mathcal{O}(\lambda)}$. Therefore, the notion of degeneration does not depend on the particular representatives, and we write $J \to J'$ instead of $\lambda \to \mu$, and $\mathcal{O}(J)$ instead of $\mathcal{O}(\lambda)$. 
We write $J \not\to J'$ to indicate that $J' \notin \overline{\mathcal{O}(J)}$.

If $J$ is represented by $\lambda \in \mathbb{L}(T)$, we say that $J$ is \textit{rigid} in $\mathbb{L}(T)$ if $\mathcal{O}(\lambda)$ is an open subset of $\mathbb{L}(T)$. A subset of a variety is called \textit{irreducible} if it cannot be written as a union of two proper closed subsets. A maximal irreducible closed subset is called an \textit{irreducible component}. In particular, $J$ is rigid in $\mathcal{S}^{m,n}$ if and only if $\overline{\mathcal{O}(\lambda)}$ is an irreducible component of $\mathbb{L}(T)$. It is a well-known fact that every affine variety admits a unique decomposition into finitely many irreducible components.


To prove a non-degeneration $J \not\to J',$ we use the standard argument from Lemma, whose proof is the same as the proof of   \cite[Lemma 1.5]{GRH}.

\begin{lemma}
Let $\mathfrak{B}$ be a Borel subgroup of ${\rm GL}(\mathbb V)$ and ${\rm R}\subset \mathbb{L}(T)$ be a $\mathfrak{B}$-stable closed subset.
If $J  \to J'$ and  the superalgebra $J $ can be represented by a structure $\mu\in{\rm R}$, then there is $\lambda\in {\rm R}$ representing $J'$.
\end{lemma}

\section{Results}

Let $(3,1)_i$ and $(2,2)_i$ denote the four-dimensional Jordan superalgebras as  in \cite{p1, p2}.

\begin{theorem}
    There is no degeneration $(3,1)_i \rightarrow (3,1)_j$  for the following index pairs $(i,j)$.

\begin{longtable}{lllllllllll}
    $(6,2);$ \  $(46,28);$ \ $(31,34);$ \ $(21,38);$ \    $(21,42);$ \ $(27,42);$ \ $(53,43);$ \ $(55,47);$\\ 
    $(i,40)$  for   $ i \in \{ 6,11, \dots,16, 18,46, 49 , \dots, 55 \}.$ \\
\end{longtable}

\end{theorem}

\begin{proof}

The following conditions prove the listed non-degenerations.
Namely, it is easy to see that $J$ satisfies ${\rm R},$ 
but $J'$ does not satisfy ${\rm R},$ that gives  
$J \not \rightarrow J'$  due to Lemma.

\begin{longtable}{lll|llllllll}

 \hline

$J$ &$ \not \rightarrow$&$J'$ &
$\begin{array}{l}
{\rm R}
\end{array}$\\

 \hline
$6$ &$ \not \rightarrow$&$  
2,\ 40 $ &
$\begin{array}{l}
c_{12}^3=c_{23}^1=c_{33}^2=c_{33}^1=0, \ c_{12}^1=c_{22}^2, \ 2c_{23}^2=c_{33}^3, \\ 
2c_{14}^4 =c_{11}^1 +c_{13}^3,\ c_{11}^3c_{22}^1=c_{22}^3 \big(c_{14}^4-c_{12}^2-c_{13}^3\big), \\  c_{11}^2c_{22}^3=-c_{11}^3c_{12}^1, \ 
 3c_{11}^3c_{13}^2=c_{11}^2\big(2c_{13}^3-c_{11}^1-c_{12}^2\big)
\end{array}
$\\ 
\hline



$11,12$ &$ \not \rightarrow$&$  
40\ $ &
$\begin{array}{l}
c_{11}^2=c_{22}^1=c_{23}^1=c_{23}^2=c_{31}^2=c_{33}^1=0, \ 
c_{11}^1=2c_{21}^2 \   
\end{array}
$\\ 
\hline

$13,14,49,50$ &$ \not \rightarrow$&$  
40\ $ &
$\begin{array}{l}
c_{11}^2=c_{13}^2=c_{22}^1=c_{32}^1=c_{32}^3=c_{33}^1=0 \ 
\end{array}
$\\ 
\hline

$15$ &$ \not \rightarrow$&$  
40\ $ &
$\begin{array}{l}
c_{13}^1=c_{22}^1=c_{22}^3=c_{23}^1=c_{31}^2=0, \ c_{13}^3c_{11}^3=c_{11}^1c_{11}^3+c_{11}^2c_{12}^3
\end{array}
$\\ 
\hline

$16$ &$ \not \rightarrow$&$  
40\ $ &
$\begin{array}{l}
c_{13}^1=c_{13}^2=c_{22}^1=c_{22}^3=c_{23}^1=0, \ c_{12}^2c_{11}^3=c_{11}^1c_{11}^3+2c_{11}^2c_{12}^3
\end{array}
$\\
\hline

$18$ &$ \not \rightarrow$&$  
40\ $ &
$\begin{array}{l}
c_{13}^2=c_{22}^1=c_{23}^1=c_{32}^2=c_{33}^1=0, \ c_{12}^2c_{11}^3=2c_{11}^2c_{12}^3
\end{array}
$\\ 
\hline


$21$ &$ \not \rightarrow$&$  
38, \ 42$ &
$\begin{array}{l}
c_{12}^1=c_{22}^1=c_{33}^1=c_{33}^2=0, \ c_{13}^3=2c_{14}^4, \\ 
c_{23}^3=2c_{24}^4, \ c_{33}^3=2c_{34}^4
\end{array}
$\\ 
\hline

$27$ &$ \not \rightarrow$&$  
42\ $ &
$\begin{array}{l}
c_{12}^1=c_{22}^1=c_{33}^3=0, \ c_{21}^2=2c_{14}^4, \ c_{22}^2=2c_{24}^4
\end{array}
$\\ 
\hline

$31$ &$ \not \rightarrow$&$  
34\ $ &
$\begin{array}{l}
c_{12}^1=c_{13}^1=c_{13}^2=c_{23}^2=c_{23}^3=0, \ c_{11}^1=c_{14}^4
\end{array}
$\\ 
\hline

$46$ &$ \not \rightarrow$&$  
28, \ 40\ $ &
$\begin{array}{l}
c_{13}^1=c_{23}^1=c_{33}^1=c_{33}^2=0, \ 
c_{11}^2c_{22}^3=-c_{11}^3c_{12}^1, \\ 
3c_{11}^3c_{22}^1=c_{22}^3 \big(c_{11}^1-2c_{12}^2-2c_{13}^3\big), \
2c_{12}^3c_{22}^1=c_{22}^3 \big(c_{12}^1-c_{22}^2\big), \\ \big(c_{11}^2\big)^2c_{33}^3=2c_{13}^2\big(c_{11}^2c_{12}^2+c_{11}^3c_{13}^2\big),\ 2c_{32}^3=c_{12}^1+c_{22}^2,  \\  2\big(c_{11}^2\big)^2c_{12}^3=c_{11}^3\big(c_{11}^2c_{12}^2-c_{11}^1c_{11}^2-c_{11}^3c_{13}^2\big)
\end{array}
$\\ 
\hline



$51$ &$ \not \rightarrow$&$  
40\ $ &
$\begin{array}{l}
c_{13}^2=c_{22}^1=c_{22}^3=c_{23}^1=0, \ c_{13}^3c_{11}^3=c_{11}^3c_{12}^2-c_{11}^2c_{12}^3
\end{array}
$\\ 
\hline

$52,\ 54,\ 55$ &$ \not \rightarrow$&$  
40\ $ &
$\begin{array}{l}
c_{13}^1=c_{13}^2=c_{22}^1=c_{22}^3=c_{23}^1=0, \ c_{12}^2c_{11}^3=2c_{11}^2c_{12}^3
\end{array}
$\\ 
\hline

$53$ &$ \not \rightarrow$&$  
40,\ 43 $ &
$\begin{array}{l}
c_{13}^1=c_{13}^2=c_{22}^1=c_{22}^3=c_{23}^1=c_{33}^1=c_{33}^3=0, \\ 
c_{13}^3c_{11}^3=c_{11}^2c_{12}^3, 
\ c_{11}^1=c_{14}^4, \ c_{22}^2=2c_{23}^3
\end{array}
$\\
\hline



$55$ &$ \not \rightarrow$&$  
47\ $ &
$\begin{array}{l}
c_{12}^1=c_{13}^1=c_{13}^2=c_{22}^1=c_{32}^1=0, \ c_{12}^2=c_{14}^4
\end{array}
$\\ 
\hline

\end{longtable}
\noindent
$c_{ij}^{k}$ coefficients are structural constants in the basis $x_1=e_1,\  x_2=e_2,\  x_3=e_3,\   x_4=f$.
\end{proof}

\begin{theorem}

    There is no degeneration $(2,2)_i \rightarrow (2,2)_j$  for the following index pairs $(i,j)$.

\begin{longtable}{lllllllllll}
    $(28,24);$ \ $(3,25);$ \  $(28,26);$ \  $(11,33);$ \  $(13,33);$ \
    \ $(7,41);$ \  $(45,41);$ \ $(45,43);$ \\
    $(i,32)$  for   $ i \in \{ 10, 12, 14, 45 \};$ \\
    $(i,47)$  for   $ i \in \{ 3,5, 7, 20, 28, 38, 45 \};$ \\
    $(i,51)$  for   $ i \in \{ 8, 11, 13 \};$ \\
    $(D_{\gamma}, j)$ for $j\in \{2,4,17,18,30,31,32, 35,36\}$.
\end{longtable}
    
\end{theorem}

\begin{proof}

The following conditions justify the listed non-degenerations.
    \begin{longtable}{lll|llll}

 \hline

$J$ &$ \not \rightarrow$&$J'$ &
$\begin{array}{l}
{\rm R}
\end{array}$\\

\hline

$3$ &$ \not \rightarrow$&$ {25},\  {47}\ $ &
$  
\begin{array}{l}
c_{23}^3=0, \ c_{22}^1=0, \  2c_{11}^2c_{34}^1=c_{34}^2 \big(2c_{11}^1-c_{12}^2\big)
\end{array}
$\\ 
\hline

$5$ &$ \not \rightarrow$& $ {47}$ &
$  
\begin{array}{l}
c_{22}^1=0, \  c_{11}^1c_{34}^2=c_{12}^2c_{34}^2+c_{11}^2c_{34}^1
\end{array}
$\\ 
\hline

$7,\ {45}$ & $ \not \rightarrow$ & ${41}, \ {47}$ &
$  
\begin{array}{l}
c_{22}^1=c_{12}^1=c_{34}^1=0
\end{array}
$\\ 
\hline

$8$ &$ \not \rightarrow$&${51}$ &
$  
\begin{array}{l}
c_{22}^1=c_{23}^4=c_{24}^3=0
\end{array}
$\\ 
\hline

${10},\ {12},\ {14}$ &$ \not \rightarrow$&${32}\ $ &
$  
\begin{array}{l}
c_{12}^1=0, \ c_{11}^1=c_{13}^3
\end{array}
$\\ 
\hline

$11$ &$ \not \rightarrow$&$  
33, \ 51 $ &
$  
\begin{array}{l}
c_{12}^1=c_{14}^3=c_{22}^1=c_{24}^3=0, \\ 
c_{12}^2=c_{14}^4, \ c_{12}^2c_{13}^4=c_{11}^2c_{23}^4
\end{array}
$\\ 
\hline

$13$ &$ \not \rightarrow$&$  
33,\ 51 $ &
$  
\begin{array}{l}
c_{12}^1=c_{14}^3=c_{22}^1=c_{23}^4=c_{24}^3=0, \ c_{11}^1=c_{13}^3
\end{array}
$\\ 
\hline
${20},\ {38}$ &$ \not \rightarrow$&${47}$ &
$  
\begin{array}{l}
c_{22}^1=c_{34}^1=0
\end{array}
$\\ 
\hline

$28$ &$ \not \rightarrow$&$  
24,\ 26,\  47 $ &
$  
\begin{array}{l}
c_{22}^1=c_{22}^2=c_{34}^1=0
\end{array}
$\\ 
\hline

$45$ &$ \not \rightarrow$&$  
32,\ 43 $ &
$  
\begin{array}{l}
c_{12}^1=c_{22}^1=c_{34}^1=0, \ c_{14}^4=2c_{11}^1-c_{13}^3
\end{array}
$\\ 
\hline

$D_\gamma$ &$ \not \rightarrow$&$ 
2,\ 4,\   17,\  18,\  30,  
 $ &
$  
\begin{array}{l}
c_{12}^1=0, \ c_{22}^2=2c_{23}^3,  
\end{array}
$\\ 

&& $31,\ 32,\ 35,\  36$ &
$\begin{array}{l}
  c_{12}^2=2c_{13}^3-c_{11}^1
\end{array}$
\\
\hline

\end{longtable}\noindent
$c_{ij}^{k}$ coefficients are structural constants in the basis $x_1=e_1,  \ x_2=e_2, \  x_3=f_1,\   x_4=f_2$.
\end{proof}

\subsubsection*{Acknowledgments} This work was supported by 
FCT 2023.08031.CEECIND and UID/00212/2025.

\EditInfo{November 13, 2025}{December 2, 2025}{Ivan Kaygorodov  and David Towers}

\end{document}